%% file: paper.tex
\documentclass[a4paper,12pt]{amsart}
\usepackage[hmarginratio={1:1},vmarginratio={1:1},heightrounded,textwidth=400pt]{geometry}
\usepackage{amsmath,amssymb,amsthm}

\usepackage{amsfonts}
\usepackage[hidelinks]{hyperref}
\usepackage[T1]{fontenc}
\usepackage{tikz}
\usetikzlibrary{calc}
\usetikzlibrary{decorations.markings}
\usetikzlibrary{arrows}
\usetikzlibrary{patterns}
\usetikzlibrary{shapes,snakes}
\usepackage[textsize=small]{todonotes}
\usepackage{url}

\newtheorem{theorem}{Theorem}
\newtheorem*{theorem*}{Theorem}

\newtheorem{lemma}[theorem]{Lemma}

\let\leq\leqslant
\let\geq\geqslant
\let\setminus\smallsetminus
\let\grminus\setminus

\newcommand{\brac}[1]{{\left(#1\right)}}

\newcommand{\set}[1]{\left\{#1\right\}}
\newcommand{\norm}[1]{{\left|#1\right|}}

\newcommand{\dir}[1]{{\vec{#1}}}

\newcommand{\Oh}[1]{O\brac{#1}}

\newcommand{\Nat}{\mathbb{N}}

\makeatletter
\newcommand\ie{i.e\@ifnextchar.{}{.\@}}
\newcommand\etc{etc\@ifnextchar.{}{.\@}}
\newcommand\etal{et~al\@ifnextchar.{}{.\@}}
\makeatother

\newcounter{hackcount}
\newcommand{\hackcounter}[1]{\setcounter{hackcount}{\value{theorem}}\setcounterref{theorem}{#1}\addtocounter{theorem}{-1}}
\newcommand{\unhackcounter}{\setcounter{theorem}{\value{hackcount}}}

\begin{document}

\title{Plane graphs are facially-non-repetitively $10^{4 \cdot 10^7}$-choosable}

\author[G.~Gutowski]{Grzegorz Gutowski}

\address[G.~Gutowski]{Theoretical Computer Science Department, Faculty of Mathematics and Computer Science, Jagiellonian University, Krak\'ow, Poland}
\email{gutowski@tcs.uj.edu.pl}
\thanks{
Partially supported by Polish National Science Center grant 2016/21/B/ST6/02165. 
}

\begin{abstract}
A sequence $\brac{x_1,x_2,\ldots,x_{2n}}$ of even length is a repetition if $\brac{x_1,\ldots,x_n} = \brac{x_{n+1},\ldots,x_{2n}}$.
We prove existence of a constant $C < 10^{4 \cdot 10^7}$ such that given any planar drawing of a graph $G$, and a list $L(v)$ of $C$ permissible colors for each vertex $v$ in $G$, there is a choice of a permissible color for each vertex such that the sequence of colors of the vertices on any facial simple path in $G$ is not a repetition.

\end{abstract}

\maketitle

\section{Introduction}

For two real-valued functions $f$ and $g$ whose domains are cofinite subsets of $\Nat$, we write $f(k) = \Oh{g(k)}$ if there exist constants $n_0$ and $c$ such that $\norm{f(k)} \leq c\norm{g(k)}$ for all $k \geq n_0$.
All graphs considered in this paper are finite, undirected and contain no loops nor multiple edges.
Given a graph $G$, a \emph{planar drawing} of $G$ is a geometric representation of $G$ in the plane such that:
\begin{itemize}
  \item each vertex $v$ is drawn as a distinct point $p_v$,
  \item each edge $\set{u,v}$ is drawn as a simple curve connecting $p_u$ and $p_v$,
  \item no two edges intersect except at their common endpoints.
\end{itemize}
A graph is \emph{planar} if it admits a planar drawing.
A planar drawing of $G$ is a \emph{straight-line} drawing of $G$ if each edge is drawn as a segment.
Let $G$ be a connected planar graph.
A planar drawing of $G$ divides the plane into topologically connected regions, called \emph{faces}.
Exactly one face is an infinite region, and is called the \emph{external face}.
Each face $F$ is described by the cyclic order of vertices of $F$ as they are visited when the boundary of $F$ is traversed in the clockwise direction.
The description of all the faces determined by a planar drawing of $G$ and choice of the external face is a \emph{planar embedding} of $G$.
Two planar drawings of $G$ are \emph{equivalent} if they both determine the same planar embedding.
A \emph{plane} graph is a non-empty connected planar graph with a fixed planar drawing.

For a sequence $X = \brac{x_1,\ldots,x_n}$, and any $1 \leq i \leq j \leq n$, a sequence $X[i,j] = \brac{x_i,\ldots,x_j}$ is a \emph{block} of $X$.
A sequence $X = \brac{x_1,x_2,\ldots,x_{2n}}$ of even length is a \emph{repetition} if $X[1,n] = X[n+1,2n]$.
A sequence $X$ is \emph{non-repetitive} if no block of $X$ is a repetition.
The study of non-repetitive sequences was initiated by Thue~\cite{Thue06} who proved that there are arbitrarily long non-repetitive sequences with only three different elements.

A \emph{proper coloring} of a graph $G$ is a coloring of the vertices of $G$ such that no two endpoints of an edge of $G$ are colored the same.
A \emph{non-repetitive coloring} of $G$ is a coloring of the vertices of $G$ such that the sequence of colors of vertices on any simple path in $G$ is not a repetition.
In particular, any non-repetitive coloring of $G$ is a proper coloring of $G$.
The study of non-repetitive colorings was initiated by Alon \etal{}~\cite{AlonGHR02}.
They conjectured that every planar graph is non-repetitively $\Oh{1}$-colorable.
Currently, the best result supporting this conjecture is by Dujmovi{\'c} \etal{}~\cite{DujmovicFJW12} who proved that every planar graph on $n$ vertices is non-repetitively $\Oh{\log n}$-colorable.

A \emph{facial walk} in a plane graph $G$ is a walk that traverses consecutive vertices of the boundary of some face of $G$, and that traverses each edge at most once.
A \emph{facial path} in $G$ is a facial walk which is a simple path.
A \emph{facially-non-repetitive coloring} of a plane graph $G$ is a coloring of the vertices of $G$ such that the sequence of colors of vertices on any facial path in $G$ is not a repetition.
Bar{\'a}t and Czap~\cite{BaratC13} proved that every plane graph is facially-non-repetitively $24$-colorable.

A \emph{list assignment} of a graph $G$ is a mapping $L$ which assigns to each vertex $v$ of $G$ a set $L(v)$ of permissible colors.
For two list assignments $L$, and $M$ of the graph $G$, we write $M \subseteq L$ to denote that $M(v) \subseteq L(v)$, for each vertex $v$ in $G$.
An \emph{$L$-coloring} of $G$ is a coloring $c$ of vertices of $G$ such that $c(v) \in L(v)$ for every vertex $v$ in $G$.
We say that a list assignment $M$ of $G$ is proper (or non-repetitive, or facially-non-repetitive) if any $M$-coloring of $G$ is a proper, (or non-repetitive, or facially-non-repetitive) coloring of $G$. 

A list assignment $L$ of $G$ is a \emph{$k$-list assignment} of $G$ if $\norm{L(v)} \geq k$, for each vertex $v$ in $G$.
A graph $G$ is properly (or non-repetitively, or facially-non-repetitively) \emph{$k$-choosable} if for any $k$-list assignment $L$ of $G$ there is a proper (or non-repetitive, or facially-non-repetitive) $L$-coloring of $G$.
Fiorenzi \etal{}~\cite{FiorenziOMZ11} showed that for any constant $C$ there is a tree which is not non-repetitively $C$-choosable.
Przyby\l{}o \etal{}~\cite{PrzybyloSS16} proved that every plane graph of maximum degree $\Delta$ is facially-non-repetitively $\Oh{\Delta}$-choosable.
In this paper we improve this result and prove that every plane graph is facially-non-repetitively $\Oh{1}$-choosable.

We say that a graph $G$ is properly (or non-repetitively, or facially-non-repetitively) \emph{$\brac{k:m}$-choosable} if for any $k$-list assignment $L$ of $G$ there is a proper, (or non-repetitive, or facially-non-repetitive) $m$-list assignment $M \subseteq L$ of $G$.

Main result proved in this paper is the following.
\begin{theorem}\label{thm:main}\leavevmode\newline
  Every plane graph is facially-non-repetitively $\brac{\Oh{m^{{43046721}}}:m}$-choosable.
\end{theorem}
The proof gives an explicit polynomial $f(m)$ of degree $3^{16}$ such that every plane graph is facially-non-repetitively $\brac{f(m):m}$-choosable.
When we compute the value of this polynomial for $m=1$ we get that every plane graph is facially-non-repetitively $10^{4 \cdot 10^7}$-choosable.

The proof is based on the following earlier results.
First ingredient is the famous Four Color Theorem.
\begin{theorem}[Appel, Haken, Koch~\cite{AppelH77,AppelHK77}]\label{thm:four_color}\leavevmode\newline
  Every plane graph is properly $4$-colorable.
\end{theorem}

Second tool is a beautiful technique by Thomassen~\cite{Thomassen94} who showed that every plane graph is properly $5$-choosable.
We use the following stronger statement which can be obtained using the same proof.
The modified proof is presented in Appendix~\ref{app:five_color}.
\begin{theorem}[Thomassen~\cite{Thomassen94}]\label{thm:five_color}\leavevmode\newline
  Every plane graph is properly $\brac{5m:m}$-choosable.
\end{theorem}

Grytczuk \etal{}~\cite{GrytczukKM13} showed that every path is non-repetitively $4$-choosable.
In our proof, we use an even stronger result by G{\k{a}}gol \etal{}~\cite[Lemma~6]{GagolJKM16}.
The proof of the following theorem uses the entropy compression method and works for graphs more general than simple paths.
\begin{theorem}[G{\k{a}}gol, Joret, Kozik, Micek~\cite{GagolJKM16}]\label{thm:path}\leavevmode\newline
  Every simple path is non-repetitively $\brac{32m^3+1:m}$-choosable.
\end{theorem}

The last ingredient is the study of bipolar orientations of planar graphs.
This notion was first used by Lempel \etal{}~\cite{LempelEC67} to develop planarity testing algorithm.
Let $G$ be a plane graph with a planar straight-line drawing such that no two vertices have the same $y$-coordinate.
Let $\dir{G}$ denote an orientation of $G$ that arises from directing each edge upward, \ie{} towards a vertex with bigger $y$-coordinate.
A \emph{source} in a directed graph is a vertex with no incoming edges.
Similarly, a \emph{sink} is a vertex with no outgoing edges.
We say that the drawing of $G$ is \emph{bipolar} if $\dir{G}$ is acyclic, has a single source $s$, and a single sink $t$.
In order to easily distinguish vertices $s$ and $t$, we call such a drawing to be \emph{$(s,t)$-bipolar}.

\begin{theorem}[Lempel, Even, Cederbaum~\cite{LempelEC67}]\label{thm:bipolar}\leavevmode\newline
  For every 2-connected plane graph $G$, and two vertices $s$, and $t$ on the external face of $G$, there is an equivalent $(s,t)$-bipolar drawing of $G$.
\end{theorem}

For a face $F$ of an $(s,t)$-bipolar drawing of $G$, let $s(F)$ -- \emph{source} of $F$, and $t(F)$ -- \emph{sink} of $F$ be the vertex with respectively the minimal, and the maximal $y$-coordinate among vertices of $F$.
Observe, that the source of the external face of $G$ is $s$, and the sink of the external face of $G$ is $t$.
Tamassia and Tollis~\cite{TamassiaT86} showed that the boundary of any face $F$ of a bipolar orientation consists of two directed paths from $s(F)$ to $t(F)$.
As a result, we get that any vertex $v$ is a source or sink of all but at most two faces.
See Figure~\ref{fig:bipolar}.

\input fig_bipolar.tex

\section{Result}

The main idea behind the coloring algorithm is the following.
We say that a vertex $v$ on a face $F$ is either \emph{regular} or \emph{special} for $F$.
We divide occurrences of the vertices on the faces so that:
\begin{itemize}
  \item there are two special vertices for any face,
  \item each vertex is regular for at most two faces.
\end{itemize}
Lemma~\ref{lem:drawing} gives such a division.
The construction is based on bipolar orientations of $2$-connected plane graphs.

Our coloring algorithm first filters list assignment so that any two vertices in distance at most two on any face have disjoint lists of colors.
This is obtained by applying Theorem~\ref{thm:five_color} a few times and is described in Lemma~\ref{lem:square}.
Then, we introduce a technique that chooses colors for a single face in a slightly augmented setting.
Intuitively, each face $F$ "controls" the colors of the regular vertices for $F$, but have to "accept" any colors of the special vertices for $F$.
This way, large faces "control" the colors of most of its vertices.
On the other hand, list of colors for any vertex is "controlled" by at most two faces.
Lemma~\ref{lem:face} gives a method to filter lists of the regular vertices for a single face $F$ so that no matter how the special vertices for $F$ are colored there is no repetition on any facial path of $F$.
The proof covers the face with paths and uses Theorem~\ref{thm:path}.
Last observation that completes the proof is that we can introduce an auxiliary planar graph on the faces of $G$, color them with four colors, and apply Lemma~\ref{lem:face} simultaneously for all faces of the same color.

\begin{lemma}[Drawing Lemma]\label{lem:drawing}
  Let $G$ be a plane graph, and $s$, $t$ be two vertices on the external face of $G$.
  There is a way to name each occurrence of a vertex on a face either normal or special and satisfy the following conditions:
  \begin{itemize}
    \item there are two special vertices for any face,
    \item each vertex is regular for at most two faces,
    \item vertices $s$ and $t$ are not regular for any face.
  \end{itemize}
\end{lemma}
\begin{proof}
  By induction on the number of vertices of $G$.
  If $G$ has two vertices $s$ and $t$, then both vertices are special for the only face of $G$.

  If $G$ is 2-connected, we apply Theorem~\ref{thm:bipolar} and get an $(s,t)$-bipolar drawing of $G$.
  For each face $F$, we choose the special vertices for $F$ to be $s(F)$ and $t(F)$.
  Properties of bipolar drawings, see Figure~\ref{fig:bipolar}, guarantee that each vertex $v$ is special for all but at most two faces.
  Vertices $s$ and $t$ are special for all the faces.

  If $G$ is not 2-connected, let $v$ be a cut-point of $G$ and $G_1,\ldots,G_k$ be the components of $G\grminus\set{v}$.
  Without loss of generality, assume that $s$ is in $G_1$, and that $t$ is either in $G_1$, or in $G_2$.
  Let $G'$ denote the graph $G \grminus G_k$, and $G''$ denote the graph $G \grminus G_1,\ldots,G_{k-1}$.
  Let $F$ be the only face of $G$ such that the boundary of $F$ has edges both in $G'$ and $G''$.
  Let $F'$ be the face of $G'$ such that $G''$ is drawn inside face $F'$.
  Let $F''$ be the external face of $G''$.
  Observe, that vertex $v$ is in $F \cap F' \cap F''$ and that the boundary of $F$ is a boundary of $F'$ with "inserted" boundary of $F''$.
  See Figure~\ref{fig:drawing}.

  First, assume that $t$ is in $G'$.
  Apply induction for $G'$, $s$, and $t$.
  Choose any vertex $w$ other than $v$ in $F''$.
  Apply induction for $G''$, $v$, and $w$.
  For each face of $G$ different than $F$, special vertices are determined by induction.
  We set the special vertices for $F$ to be the special vertices for $F'$ in $G'$.
  Observe that vertex $w$ is special for all faces of $G''$, and thus $w$ is normal only for one face $F$ of $G$.
  Every vertex other than $w$ that is normal for $F$ is either normal for $F'$, or normal for $F''$.

  Now, assume that $t$ is in $G''$.
  This means that $k=2$, $t$ is in $G_2$ and $F$ is the external face of $G$.
  Apply induction for $G'$, $s$, and $v$.
  Apply induction for $G''$, $v$, and $t$.
  For each face of $G$ different than $F$, special vertices are determined by two induction calls.
  We set the special vertices for $F$ to be $s$ and $t$.
  Vertex $v$ is special for all faces of $G'$, and for all faces of $G''$, and thus $v$ is normal only for one face $F$ of $G$.
  Every vertex other than $v$ that is normal for $F$ is either normal for $F'$, or normal for $F''$.
\end{proof}

\input fig_drawing.tex

Now, we prove several lemmas, that allow us to \emph{filter} list assignments of different structures.
Each of those lemmas tells us that some plane graphs are $\brac{f(m):m}$-choosable for a polynomial function $f$.
For example, set
$$f_1(m) = 5m\text{.}$$
Theorem~\ref{thm:five_color} gives that every planar graph is properly $\brac{f_1(m):m}$-choosable.

For a plane graph $G$, a \emph{facial-square} of $G$ is a graph on the same vertex set in which two vertices $u$ and $v$ are connected by an edge when $u$ and $v$ are connected in $G$ by a facial path of length at most two.
We say that a coloring of a plane graph $G$ is a \emph{facially-square-proper} if it is a proper coloring of the facial-square of $G$.

For the next lemma, set
$$f_2(m) = f_1(f_1(f_1(m))) = 125m\text{.}$$
\begin{lemma}[Facial-Square Filtering Lemma]\label{lem:square}
  Facial-square of a plane graph is properly $\brac{f_2(m) : m}$-choosable.
\end{lemma}
\begin{proof}
  Let $G$ be a plane graph, and let $L$ be an $f_2(m)$-list assignment of $G$.

  We show that edges of facial-square of $G$ can be decomposed into three sets, say red, green, and blue so that any two edges of the same color are non-crossing.
  See Figure~\ref{fig:square} for an example of the following decomposition.
  First, color red all edges of $G$.
  Then, for each face $F$, we color the edges that correspond to pairs of vertices in distance two on $F$.
  Let $l$ denote the number of vertex occurrences on the boundary of $F$, and let $v_0,v_1,\ldots,v_{l-1}$ be the vertices of $F$ in the clockwise order.
  Color green every edge $\set{v_i, v_j}$, where $j = \brac{\brac{i+2} \mod l}$ and both $i$ and $j$ are odd numbers.
  Similarly, color green every edge $\set{v_i, v_j}$, where $j = \brac{\brac{i+2} \mod l}$ and both $i$ and $j$ are even numbers.
  If $l$ is odd, color green the edge $\set{v_1,v_{l-1}}$, and color red the edge $\set{v_0,v_{l-2}}$.

  We apply Theorem~\ref{thm:five_color} to the red graph and obtain a $f_1(f_1(m))$-list assignment of the facial-square of $G$ in which any two vertices connected by a red edge have disjoint lists of permissible colors.
  We repeat the same two more times, for green, and for blue edges.
  In the end we obtain a facially-square-proper $m$-list assignment of $G$.
\end{proof}

\input fig_square.tex

For the next lemma, set
$$
\begin{array}{lcl}
  f_3(m) & = & 32m^3+1\text{,}\\
  f_4(m) & = & f_3(m)+m = 32m^3+m+1\text{,}\\
  f_5(m) & = & f_3(f_4(m)) + m + f_4(m)= \Oh{m^{9}}\text{.}\\
\end{array}
$$
Theorem~\ref{thm:path} gives that a path is non-repetitively $\brac{f_3(m):m}$-choosable.
\begin{lemma}[Path Filtering Lemma]\label{lem:path}
  Let $m \geq 1$.
  Let $P$ be a simple path $\brac{v_1,\ldots,v_n}$, and $v_s$ be a selected vertex in $P$.
  Let $L$ be a proper list assignment of $P$ such that:
  \begin{itemize}
    \item $\norm{L(v_i)} = f_4(m)$, for $i \leq s$,
    \item $\norm{L(v_i)} = f_5(m)$, for $i > s$.
  \end{itemize}
  There is a non-repetitive list assignment $M \subseteq L$ of $P$ such that:
  \begin{itemize}
    \item $\norm{M(v_i)} = m$, for $i < s$,
    \item $\norm{M(v_i)} = f_4(m)$, for $i \geq s$.
  \end{itemize}
\end{lemma}
\begin{proof}
  Set $M(v_s) = L(v_s)$.
  If $s>1$, then choose any $m$ colors for $M(v_{s-1})$ from $L(v_{s-1})$.
  Next, for each $i < s-1$, set $L'(v_i) = L(v_i) \setminus M(v_{s-1})$.
  $L'$ is a proper $f_3(m)$-list assignment of $\brac{v_1,\ldots,v_{s-2}}$.
  Apply Theorem~\ref{thm:path} to $L'$.
  The resulting $m$-list assignment is non-repetitive and uses the set of colors disjoint with $M(v_{s-1})$.
  Thus, we get a non-repetitive list assignment $M$ of $\brac{v_1,\ldots,v_{s-1}}$.

  Then, for each $i>s$, set $L''(v_i) = L(v_i) \setminus \brac{M(v_s) \cup M(v_{s-1})}$, or $L''(v_i) = L(v_i) \setminus M(v_s)$ if $s=1$.
  $L''$ is a proper $f_3(f_4(m))$-list assignment of $\brac{v_{s+1},\ldots,v_n}$.
  Apply Theorem~\ref{thm:path} to $L''$.
  The resulting $f_4(m)$-list assignment is non-repetitive and uses the set of colors disjoint with $M(v_s)$.
  Thus, we get a non-repetitive list assignment $M$ of $\brac{v_{s},\ldots,v_{n}}$.

  Suppose, that there is a repetition in $M$.
  Such a repetition must include vertex $v_{s-1}$.
  Observe that $M(v_{s-1})$ and $M(v_s)$ are disjoint as $L$ is a proper list assignment.
  The set of colors $M(v_{s-1})$ was removed from the list of permissible colors for every other vertex.
  Thus, the color of $v_{s-1}$ is not repeated.
  A contradiction.
\end{proof}

For a walk $W$ in a graph, a \emph{simple $W$-block} is a block of $W$ that is a simple path.
For the next lemma, set
$$f_6(m)=f_5(f_5(m))=\Oh{m^{81}}\text{.}$$
\begin{lemma}[Walk Filtering Lemma]\label{lem:walk}
  Let $W$ be a facial walk in the graph, and let $L$ be a proper $f_6(m)$-list assignment of $W$.
  There is an $m$-list assignment $M \subseteq L$ of $W$ such that for any simple $W$-block $P$, $M$ is a non-repetitive list assignment of $P$.
\end{lemma}

\begin{proof}
  Observe, that the repeated occurrences of vertices in $W$ have the following laminar structure.
  Let $w_i = w_j = v$, and let $u$ be any vertex in $W$ other than $v$.
  Then, either all occurrences of $u$ are in $W[i+1,j-1]$, or all occurrences of $u$ are in $W[1,i-1] \cup W[j+1,n]$.
  Assume to the contrary $w_a = w_b = u$ and $i < a < j < b$.
  Each of the three walks $W[i,a]$, $W[a,j]$, $W[j,b]$ connects $u$ and $v$.
  These three walks divide the plane into at least three regions, and thus $W$ is not a facial walk.

  We say that a path $P$ is a \emph{maximal simple $W$-block} if $P$ is a simple $W$-block and $P$ cannot be extended into either direction in $W$ and remain a simple block.
  Let $\mathcal{P}$ denote the set of all maximal simple $W$-blocks.
  We show that $\mathcal{P}$ can be decomposed into two sets, say red and green so that any two paths of the same color are non-overlapping $W$-blocks.
  Let $\mathcal{P} = \set{P_1=W[l_1,r_1], \ldots, P_k=W[l_k,r_k]}$ and without loss of generality $l_1 < l_2 < \ldots < l_k$.
  It follows from the maximality of each path that $r_1 < r_2 < \ldots < r_k$.
  See Figure~\ref{fig:walk}.

  For a path $P_i$, we know that $P_i$ cannot be extended to the right, so either $i=k$ and $r_i = n$, or the vertex $w_{r_i+1}$ is already on path $P_i$.
  Assume the second case and let $j$ be such that $w_j = w_{r_i+1}$ and $l_i \leq j < r_i$.
  The fact that $P_i$ is a simple path, and the laminar structure of repeated occurrences guarantee that each vertex $w_{j+1},\ldots,w_{r_i}$ has no other occurrences in $W$.
  Thus, any simple $W$-block that includes $w_j$ does not include $w_{r_i+1}$ and any simple $W$-block that includes $w_{r_i+1}$ can be extended to the left so that it includes $w_{j+1}$.
  Hence, $l_{i+1} = j+1$.
  Further, if $i+2 \leq k$, we can repeat the same reasoning for $P_{i+1}$ and get that $l_{i+2}$ is to the right of $j+1$ and to the right of some vertex with repeated occurrence.
  As vertices $w_{j+1},\ldots,w_{r_i}$ have no other occurrences in $W$, we get that $l_{i+2} > r_i$.
  Thus, paths $P_i$ and $P_{i+2}$ are non-overlapping $W$-blocks.
  We color red each path $P_i$ for $i$ odd, and color green each path $P_i$ for $i$ even.

  Let $t=f_5(m)$.
  Now, we filter list assignment $L$ so that we get a $t$-list assignment in which there is no repetition on any subpath of a red path.
  We process paths $P_1, P_3, \ldots$, one by one in this order.
  We will keep the following invariants that are satisfied after processing each path $P_i$:
  \begin{itemize}
    \item each vertex $v$ that has no occurrences on paths $P_1,\ldots,P_i$ has at least $f_6(m) = f_5(t)$ permissible colors;
    \item each vertex $v$ that has at least one occurrence on paths $P_1,\ldots,P_i$ and at least one occurrence on paths $P_{i+1},\ldots,P_k$ has at least $f_4(t)$ permissible colors;
    \item each vertex $v$ that has no occurrences on paths $P_{i+1},\ldots,P_k$ has at least $t$ permissible colors.
  \end{itemize}
  When processing path $P_i=\brac{w_{l_i},\ldots,w_{r_i}}$, we divide vertices in $P_i$ into four disjoint sets:
  \begin{itemize}
    \item $Only$ -- a vertex $v$ is in $Only$ if there are no occurrences of $v$ on paths other than $P_i$.
    \item $First$ -- a vertex $v$ is in $First$ if there are no occurrences of $v$ on paths $P_{1},\ldots,P_{i-1}$, and at least one on paths $P_{i+1},\ldots,P_k$.
    \item $Middle$ -- a vertex $v$ is in $Middle$ if there is at least one occurrence of $v$ on paths $P_1,\ldots,P_{i-1}$, and at least one on paths $P_{i+1},\ldots,P_k$.
    \item $Last$ -- a vertex $v$ is in $Last$ if there is at least one occurrence of $v$ on paths $P_{1}\ldots,P_{i-1}$, and no occurrences on paths $P_{i+1},\ldots,P_k$.
  \end{itemize}
  Our invariants guarantee that each vertex in $Only \cup First$ has at least $f_5(t)$ permissible colors,
  and that each vertex in $Middle \cup Last$ has at least $f_4(t)$ permissible colors.
  The fact that $P_i$ is a simple path, and the laminar structure of repeated occurrences guarantee that:
  \begin{itemize}
    \item There is at most one vertex in $Middle$.
    \item Every vertex in $Last$ is before any vertex in $Middle \cup First$.
    \item Every vertex in $First$ is after any vertex in $Middle \cup Last$.
  \end{itemize}

  Set vertex $s$ to be the last vertex in $Middle \cup Last \cup \set{w_{l_i}}$.
  Our invariants guarantee that:
  \begin{itemize}
    \item each vertex to the left of $s$ has at least $f_4(t)$ colors.
    \item $s$ has at least $f_4(t)$ colors.
    \item each vertex to the right of $s$ has at least $f_5(t)$ colors.
  \end{itemize}
  We apply Lemma~\ref{lem:path} to path $P_i$ with special vertex $s$ and get new list assignment.
  We get that each vertex in $Last$ keeps at least $t$ colors.
  Vertex in $Middle$, if it exists, keeps at least $f_4(t)$ colors.
  Any vertex in $First$ keeps at least $f_4(t)$ colors.
  Thus, our invariants are satisfied after $P_i$ is processed.
  When we process all red paths, we get an $f_5(m)$-list assignment of $W$ in which there is no repetition on any simple path, a subpath of a red maximal simple $W$-block.
  To finish the proof, we repeat the same process for $t=m$ and green maximal simple $W$-blocks.
\end{proof}

\input fig_walk.tex

For the next lemma, set
$$f_7(m)=f_6(m) + 10m=\Oh{m^{81}}\text{.}$$
\begin{lemma}[Face Filtering Lemma]\label{lem:face}
  Let $L$ be a facially-square-proper $f_7(m)$-list assignment of a face $F$, and let $s$ and $t$ be the special vertices for $F$.
  There is a non-repetitive list assignment $M \subseteq L$ of $F$ such that:
  \begin{itemize}
    \item $M(s) = L(s)$, and $M(t) = L(t)$,
    \item $\norm{M(v)} \geq m$, for each $v$ in $F$.
  \end{itemize}
\end{lemma}
\begin{proof}
  Let $A$ denote the set of all regular neighbors of $s$ and all regular neighbors of $t$ in $F$.
  For any vertex $v$ in $F$ let $A_5(v)$ denote the five vertices in $A$ encountered first when traversing $F$ in clockwise, and in counter-clockwise direction from $v$.
  Each set $A_5(v)$ has at most ten elements.
  See Figure~\ref{fig:face}.

  First we define $M$ for vertices in $A$.
  For each $v \in A$ we choose any $m$ colors from $L(v)$ so that $M(v)$ is disjoint with $M(w)$ for any $w \in A_5(v)$.
  As for each $v \in A$, the size of $A_5(v)$ is at most ten, and $\norm{L(v)} \geq 11m$, this can be easily done.
  Now, for each regular vertex $v$ not in $A$, remove from $L(v)$ the colors $M(w)$ for each $w \in A_5(v)$.
  This removes at most $10m$ colors from each list $L(v)$.
  The resulting list assignment $L'$ has $f_6(m)$ colors for each regular vertex not in $A$.
  Further, set $L'(v) = M(v)$ for $v$ in $A \cup \set{s,t}$.

  Let $P$ be a facial path in $F$ such that there exists an $L'$-coloring $c$ of $P$ such that the sequence of colors of vertices in $P$ is a repetition.
  Observe that, as $L$ is a facially-square-proper list assignment of $F$, we have that $P$ has at least six vertices.
  Let $A(P)$ denote the vertices in $P \cap A$, and let $a(P) = \norm{A(P)}$.
  First, observe that $P$ is a simple path and that between any three vertices in $A(P)$ there is at least one special vertex.
  Thus we have that $a(P) \leq 6$.
  Furthermore, if $a(P) = 6$ then both endpoints of $P$ are in $A$.
  Thus, we have $A(P) \subseteq A_5(v)$ for any regular vertex $v$ in $P$.
  Hence, only a special vertex can match the color of a vertex in $A(P)$ and we have that $a(P) \leq 2$, as there are only two special vertices.

  If $a(P) = 2$, then $P$ contains exactly two vertices $a_1, a_2$ in $A$, and exactly two special vertices.
  If both special vertices are in the first half of $P$, then, as there are at least three vertices in the first half, there is at least one vertex $a$ in $A$ in the first half of $P$ and colors of $a$ cannot be matched in the second half.
  Thus, in each half of $P$ there is exactly one special vertex, and exactly one vertex in A.
  Without loss of generality, assume that vertices $s$ and $a_1$ are in the first half of $P$.
  If $s$ is neither the first, nor the last vertex in the first half of $P$, then there are two vertices in $A$ in that half of $P$ -- one to the left, and one to the right of $s$.
  Similarly, $t$ is either the first, or the last vertex in the second half of $P$.
  On the other hand, we get that $a_1$, which is next to $s$ in the first half, is neither the first, nor the last vertex in the first half of $P$.
  Thus, color of $a_1$ cannot be matched by color of $t$.
  A contradiction.

  If $a(P) = 1$, then let $A = \set{a_1}$ and assume that $a_1$ is in the first half of $P$.
  Color of $a_1$ is matched by the color of a special vertex $w$ in the second half of $P$.
  As the last vertex in $P$ before $w$ is in $A$, we have that $a_1$ and $w$ are neighbors in $G$.
  As $L$ is a proper list assignment of $F$, we have that $L(a_1)$ and $L(w)$ are disjoint.
  A contradiction.

  Thus, we have that any simple path with a repetition of colors in $L'$ does not include neither a vertex in $A$, nor a special vertex.
  We will further filter list assignment $L'$ to remove all such repetitions.
  Let $W_1, \ldots, W_k$ be the connected components of the boundary of $F$ with vertices in $A \cup \set{s,t}$ removed.
  Observe that walks $W_1, \ldots, W_k$ are pairwise vertex disjoint.
  List assignment $L'$ is a proper $f_6(m)$-list assignment of each walk $W_1,\ldots,W_k$.
  We apply Lemma~\ref{lem:walk} to $L'$ and each walk $W_1, \ldots, W_k$ independently.
  We obtain a facially-non-repetitive list assignment $M$ of $F$.
\end{proof}

\input fig_face.tex

Now we combine Lemmas~\ref{lem:drawing}, \ref{lem:square}, and~\ref{lem:face} to prove the main theorem of this paper.
For the proof of the main result, set
$$f_8(m)=f_2(f_7(f_7(f_7(f_7(m)))))=\Oh{m^{43046721}}\text{.}$$

\hackcounter{thm:main}
\begin{theorem}\leavevmode\newline
  Every plane graph is facially-non-repetitively $\brac{\Oh{m^{{43046721}}}:m}$-choosable.
\end{theorem}
\unhackcounter
\begin{proof}
  Let $G$ be a plane graph and $L$ be an $f_8(m)$-list assignment of $G$.
  First, we use Lemma~\ref{lem:drawing} to divide occurrences of the vertices on the faces into regular and special.
  Next, we apply Lemma~\ref{lem:square} to obtain a facially-square-proper $f_7(f_7(f_7(f_7(m))))$-list assignment of $G$.
  Then, we construct an auxiliary graph $H$ on faces of $G$ in which we put an edge between two faces $F_1$, $F_2$ of $G$ when there is a vertex $v$ in $G$ that is regular both for $F_1$ and for $F_2$.
  Observe that graph $H$ is planar.
  Indeed, we can construct a planar drawing of $H$ from planar drawing of $G$ by placing a vertex corresponding to face $F$ anywhere in the face $F$ and routing an edge $\set{F_1,F_2}$ through the vertex that is regular both for $F_1$ and $F_2$.
  Theorem~\ref{thm:four_color} gives a proper coloring of $H$ with colors red, green, blue and yellow.
  Observe that, as no two red faces share a common regular vertex, we can apply Lemma~\ref{lem:face} to $L'$ and all red faces of $G$ simultaneously.
  We obtain an $f_7(f_7(f_7(m)))$-list assignment of $G$ in which there is no repetition on any facial path of a red face.
  We repeat the same three more times, for green, blue, and yellow faces.
  In the end we obtain a facially-non-repetitive $m$-list assignment $M$ of $G$.
\end{proof}

\section{Discussion}
  In order to get the final result, we need to compute value $f_8(1)$.
  A computer calculation gave us the value with approximately 33 million decimal digits.
  Thus, we get that any plane graph is facially-non-repetitively $C$-choosable for $C = 10^{4 \cdot 10^7}$.
  The presented proof is far from optimal and the polynomials in some of the lemmas can be improved at the expense of a more technical argument.
  Nevertheless, these improvements do not lead to any reasonable value $C$.
  We do not know any non-trivial lower bounds for $C$.

\bibliographystyle{plain}
\bibliography{paper}

\appendix
\section{Thomassen's proof}\label{app:five_color}

We present a slightly modified version of Thomassen's~\cite{Thomassen94} proof that gives the following statement.
We use almost the exact same wording as in the original proof to make it obvious that the same proof works.

\begin{theorem*}
  Let $G$ be a near-triangulation; \ie{}, $G$ is a planar graph which has no loops or multiple edges and which consists of a cycle $C$: $v_1 v_2 \cdots v_p v_1$, and vertices and edges inside $C$ such that each bounded face is bounded by a triangle.
  Assume that $v_1$ and $v_2$ are colored $\set{1,\ldots,m}$ and $\set{m+1,\ldots,2m}$, respectively, and that $L(v)$ is a list of at least $3m$ colors if $v \in C - \set{v_1,v_2}$ and at least $5m$ colors if $v \in G-C$.
  Then the coloring of $v_1$ and $v_2$ can be extended to a list $m$-coloring of $G$.
\end{theorem*}
\begin{proof}
  (by induction on the number of vertices of $G$).
  If $p=3$ and $G=C$ there is nothing to prove.
  So we proceed to the induction step.

  If $C$ has a chord $v_i v_j$ where $2 \leq i \leq j-2 \leq p-1$ ($v_{p+1} = v_1$), then we apply the induction hypothesis to the cycle $v_1 v_2 \cdots v_i v_j v_{j+1} \cdots v_1$ and its interior and then to $v_j v_i v_{i+1} \cdots v_{j-1} v_j$ and its interior.
  So we can assume that $C$ has no chord.

  Let $v_1, u_1, u_2, \ldots, u_m, v_{p-1}$ be the neighbors of $v_p$ in that clockwise order around $v_p$.
  As the interior of $C$ is triangulated, $G$ contains the path $P: v_1 u_1 u_2 \cdots u_m v_{p-1}$.
  As $C$ is chordless, $P \cup (C-v_p)$ is a cycle $C'$.
  Let $X$ be a set of $2m$ distinct colors in $L(v_p)\setminus\set{1,\ldots,m}$.
  Now define $L'(u_i) = L(u_i)\setminus X$ for $1 \leq i \leq m$ and $L'(v)=L(v)$ if $v$ is a vertex of $G$ not in $\set{u_1, u_2, \ldots, u_m}$.
  Then we apply the induction hypothesis to $C'$ and its interior and the new list $L'$.
  We complete the coloring by assigning $m$ colors from $X$ to $v_p$ such that $v_p$ and $v_{p-1}$ get disjoint sets of colors.
\end{proof}
\end{document}

%% file: fig_bipolar.tex
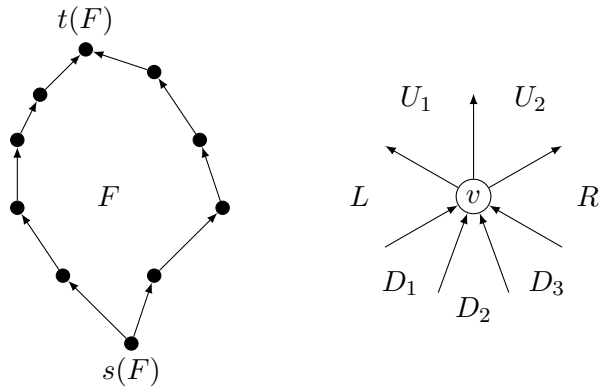
\begin{figure}[h]
  \begin{center}
\begin{tikzpicture}[>=latex, scale=1.5]
\begin{small}

\begin{scope}
\tikzstyle{every node}=[circle,minimum size=5pt,inner sep=0pt,draw,fill,label distance=1pt]

\node[draw=none,fill=none] (sl) at (1, -0.25) {$s(F)$};
\node[draw=none,fill=none] (sl) at (0.6, 2.85) {$t(F)$};
\node (s) at (1, 0) {};
\node (t) at (0.6, 2.6) {};

\node[draw=none,fill=none] (l) at (0.8, 1.3) {$F$};

\node (l1) at (0.4, 0.6) {};
\node (l2) at (0, 1.2) {};
\node (l3) at (0, 1.8) {};
\node (l4) at (0.2, 2.2) {};

\node (r1) at (1.2, 0.6) {};
\node (r2) at (1.8, 1.2) {};
\node (r3) at (1.6, 1.8) {};
\node (r4) at (1.2, 2.4) {};

\path (s) edge[->] (l1);
\path (l1) edge[->] (l2);
\path (l2) edge[->] (l3);
\path (l3) edge[->] (l4);
\path (l4) edge[->] (t);
\path (s) edge[->] (r1);
\path (r1) edge[->] (r2);
\path (r2) edge[->] (r3);
\path (r3) edge[->] (r4);
\path (r4) edge[->] (t);

\end{scope}

\begin{scope}[shift={(3,0.3)}]

\node[circle,minimum size=5pt,inner sep=2pt,draw,label distance=1pt] (x) at (1,1) {$v$};
\node (a1) at (0.133975,1.5) {};
\node (a2) at (0.357212,1.76604) {};
\node (a25) at (0.5,1.866030) {$U_1$};
\node (a3) at (0.65798,1.93969) {};
\node (a4) at (1,2) {};
\node (a5) at (1.34202,1.93969) {};
\node (a55) at (1.5, 1.86603) {$U_2$};
\node (a6) at (1.64279,1.76604) {};
\node (a7) at (1.86603,1.5) {};

\node (b1) at (1.86603,0.5) {};
\node (b2) at (1.64279,0.233956) {$D_3$};
\node (b3) at (1.34202,0.0603074) {};
\node (b4) at (1,0) {$D_2$};
\node (b5) at (0.65798,0.0603074) {};
\node (b6) at (0.357212,0.233956) {$D_1$};
\node (b7) at (0.133975,0.5) {};

\path (x) edge[->] (a1);
\path (x) edge[->] (a4);
\path (x) edge[->] (a7);
\path (x) edge[<-] (b1);
\path (x) edge[<-] (b3);
\path (x) edge[<-] (b5);
\path (x) edge[<-] (b7);

  \node (l) at (0,1) {$L$};
  \node (r) at (2,1) {$R$};

\end{scope}
\end{small}
\end{tikzpicture}
\end{center}
\caption{
  Bipolar drawing.
  On the left, boundary of face $F$ consists of two directed paths from $s(F)$ to $t(F)$.
  On the right, vertex $v$ is a source/sink of all but two faces, i.e.\@ $v = s(U_1) = s(U_2) = t(D_1) = t(D_2) = t(D_3)$.
}
\label{fig:bipolar}
\end{figure}

%% file: fig_drawing.tex
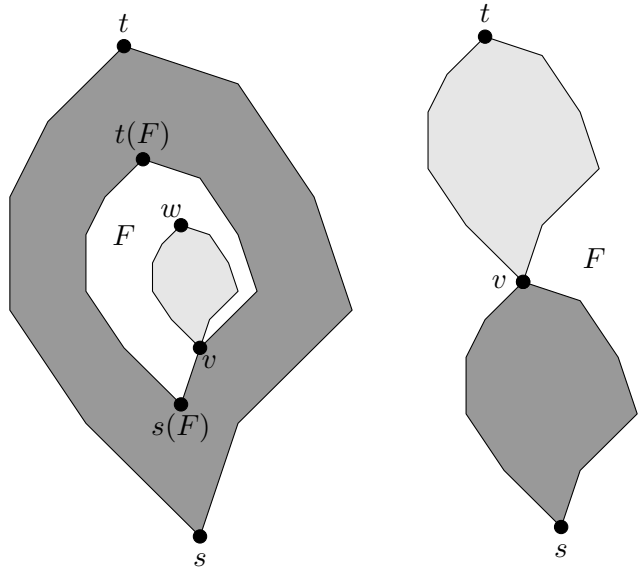
\begin{figure}[h]
  \begin{center}
\begin{tikzpicture}[>=latex, scale=1.25]
\begin{small}

\begin{scope}
\tikzstyle{every node}=[circle,minimum size=5pt,inner sep=0pt,draw,fill,label distance=1pt]
\def\scale{2}
\node[draw=none,fill=none] (a) at ($(1.2,-1.4)-{\scale}*(1,0)$) {};
\coordinate (s) at ($(a)+{\scale}*(1, 0)$) {};
\coordinate (t) at ($(a)+{\scale}*(0.6, 2.6)$) {};
\coordinate (l1) at ($(a)+{\scale}*(0.4, 0.6)$) {};
\coordinate (l2) at ($(a)+{\scale}*(0, 1.2)$) {};
\coordinate (l3) at ($(a)+{\scale}*(0, 1.8)$) {};
\coordinate (l4) at ($(a)+{\scale}*(0.2, 2.2)$) {};
\coordinate (r1) at ($(a)+{\scale}*(1.2, 0.6)$) {};
\coordinate (r2) at ($(a)+{\scale}*(1.8, 1.2)$) {};
\coordinate (r3) at ($(a)+{\scale}*(1.6, 1.8)$) {};
\coordinate (r4) at ($(a)+{\scale}*(1.2, 2.4)$) {};
\filldraw[fill=gray!80,draw=black,line join=round] (s) -- (l1) -- (l2) -- (l3) -- (l4) -- (t) -- (r4) -- (r3) -- (r2) -- (r1) -- cycle;
\node (rs) at (s) {};
\node (rt) at (t) {};
\node[draw=none,fill=none] (sl) at ($(s)+(0, -0.25)$) {$s$};
\node[draw=none,fill=none] (tl) at ($(t)+(0, 0.25)$) {$t$};
\end{scope}

\begin{scope}
\tikzstyle{every node}=[circle,minimum size=5pt,inner sep=0pt,draw,fill,label distance=1pt]
\coordinate (s) at (1, 0) {};
\coordinate (t) at (0.6, 2.6) {};
\coordinate (l1) at (0.4, 0.6) {};
\coordinate (l2) at (0, 1.2) {};
\coordinate (l3) at (0, 1.8) {};
\coordinate (l4) at (0.2, 2.2) {};
\coordinate (r1) at (1.2, 0.6) {};
\coordinate (r2) at (1.8, 1.2) {};
\coordinate (r3) at (1.6, 1.8) {};
\coordinate (r4) at (1.2, 2.4) {};
\filldraw[fill=white,draw=black,line join=round] (s) -- (l1) -- (l2) -- (l3) -- (l4) -- (t) -- (r4) -- (r3) -- (r2) -- (r1) -- cycle;
\node (rs) at (s) {};
\node (rt) at (t) {};
\node[draw=none,fill=none] (sl) at ($(s)+(0, -0.25)$) {$s(F)$};
\node[draw=none,fill=none] (tl) at ($(t)+(0, 0.25)$) {$t(F)$};
\node[draw=none,fill=none] (l) at (0.4, 1.8) {$F$};
\end{scope}

\begin{scope}
\tikzstyle{every node}=[circle,minimum size=5pt,inner sep=0pt,draw,fill,label distance=1pt]
\def\scale{0.5}
\coordinate[draw=none,fill=none] (a) at ($(1.2,0.6)-{\scale}*(1,0)$) {};
\coordinate (s) at ($(a)+{\scale}*(1, 0)$) {};
\coordinate (t) at ($(a)+{\scale}*(0.6, 2.6)$) {};
\coordinate (l1) at ($(a)+{\scale}*(0.4, 0.6)$) {};
\coordinate (l2) at ($(a)+{\scale}*(0, 1.2)$) {};
\coordinate (l3) at ($(a)+{\scale}*(0, 1.8)$) {};
\coordinate (l4) at ($(a)+{\scale}*(0.2, 2.2)$) {};
\coordinate (r1) at ($(a)+{\scale}*(1.2, 0.6)$) {};
\coordinate (r2) at ($(a)+{\scale}*(1.8, 1.2)$) {};
\coordinate (r3) at ($(a)+{\scale}*(1.6, 1.8)$) {};
\coordinate (r4) at ($(a)+{\scale}*(1.2, 2.4)$) {};
\filldraw[fill=gray!20,draw=black,line join=round] (s) -- (l1) -- (l2) -- (l3) -- (l4) -- (t) -- (r4) -- (r3) -- (r2) -- (r1) -- cycle;
\node (rs) at (s) {};
\node (rt) at (t) {};
\node[draw=none,fill=none] (sl) at ($(s)+(0.1, -0.15)$) {$v$};
\node[draw=none,fill=none] (tl) at ($(t)+(-0.1, 0.15)$) {$w$};
\end{scope}

\begin{scope}[shift={(5,0)}]

\begin{scope}
\tikzstyle{every node}=[circle,minimum size=5pt,inner sep=0pt,draw,fill,label distance=1pt]
\def\scale{1}
\node[draw=none,fill=none] (a) at ($(0,-1.3)-{\scale}*(1,0)$) {};
\coordinate (s) at ($(a)+{\scale}*(1, 0)$) {};
\coordinate (t) at ($(a)+{\scale}*(0.6, 2.6)$) {};
\coordinate (l1) at ($(a)+{\scale}*(0.4, 0.6)$) {};
\coordinate (l2) at ($(a)+{\scale}*(0, 1.2)$) {};
\coordinate (l3) at ($(a)+{\scale}*(0, 1.8)$) {};
\coordinate (l4) at ($(a)+{\scale}*(0.2, 2.2)$) {};
\coordinate (r1) at ($(a)+{\scale}*(1.2, 0.6)$) {};
\coordinate (r2) at ($(a)+{\scale}*(1.8, 1.2)$) {};
\coordinate (r3) at ($(a)+{\scale}*(1.6, 1.8)$) {};
\coordinate (r4) at ($(a)+{\scale}*(1.2, 2.4)$) {};
\filldraw[fill=gray!80,draw=black,line join=round] (s) -- (l1) -- (l2) -- (l3) -- (l4) -- (t) -- (r4) -- (r3) -- (r2) -- (r1) -- cycle;
\node (rs) at (s) {};
\node (rt) at (t) {};
\node[draw=none,fill=none] (sl) at ($(s)+(0, -0.25)$) {$s$};
\end{scope}
\begin{scope}
\tikzstyle{every node}=[circle,minimum size=5pt,inner sep=0pt,draw,fill,label distance=1pt]
\def\scale{1}
\node[draw=none,fill=none] (a) at ($(-0.4,1.3)-{\scale}*(1,0)$) {};
\coordinate (s) at ($(a)+{\scale}*(1, 0)$) {};
\coordinate (t) at ($(a)+{\scale}*(0.6, 2.6)$) {};
\coordinate (l1) at ($(a)+{\scale}*(0.4, 0.6)$) {};
\coordinate (l2) at ($(a)+{\scale}*(0, 1.2)$) {};
\coordinate (l3) at ($(a)+{\scale}*(0, 1.8)$) {};
\coordinate (l4) at ($(a)+{\scale}*(0.2, 2.2)$) {};
\coordinate (r1) at ($(a)+{\scale}*(1.2, 0.6)$) {};
\coordinate (r2) at ($(a)+{\scale}*(1.8, 1.2)$) {};
\coordinate (r3) at ($(a)+{\scale}*(1.6, 1.8)$) {};
\coordinate (r4) at ($(a)+{\scale}*(1.2, 2.4)$) {};
\filldraw[fill=gray!20,draw=black,line join=round] (s) -- (l1) -- (l2) -- (l3) -- (l4) -- (t) -- (r4) -- (r3) -- (r2) -- (r1) -- cycle;
\node (rs) at (s) {};
\node (rt) at (t) {};
\node[draw=none,fill=none] (sl) at ($(s)+(-0.25, 0)$) {$v$};
\node[draw=none,fill=none] (tl) at ($(t)+(0, 0.25)$) {$t$};
\node[draw=none,fill=none] (l) at ($(s)+(0.75, 0.25)$) {$F$};
\end{scope}

\end{scope}
\end{small}
\end{tikzpicture}
\end{center}
\caption{
  Division of a graph $G$ in Lemma~\ref{lem:drawing} into $G'$ (dark gray) and $G''$ (light gray).
  On the left, vertex $t$ is in $G'$.
  On the right, vertex $t$ is in $G''$.
}
\label{fig:drawing}
\end{figure}

%% file: fig_square.tex
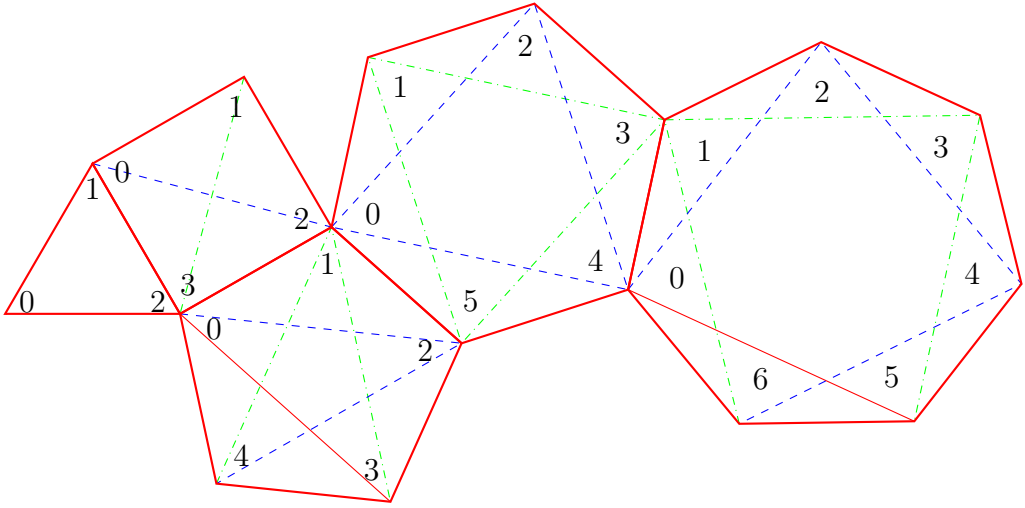
\begin{figure}[h]
  \begin{center}
\begin{tikzpicture}[scale=2.3]

  \def\rada{1/(2*sin(180/3))}
  \def\apoa{\rada*cos(180/3)}
  \def\radb{1/(2*sin(180/4))}
  \def\apob{\radb*cos(180/4)}
  \def\apoab{((\apoa+\apob)/(\apoa))}
  \def\radc{1/(2*sin(180/5))}
  \def\apoc{\radc*cos(180/5)}
  \def\apobc{((\apob+\apoc)/(\apob))}
  \def\radd{1/(2*sin(180/6))}
  \def\apod{\radd*cos(180/6)}
  \def\apocd{((\apoc+\apod)/(\apoc))}
  \def\rade{1/(2*sin(180/7))}
  \def\apoe{\rade*cos(180/7)}
  \def\apode{((\apod+\apoe)/(\apod))}

  \coordinate (ac) at (0,0) {};
  \coordinate (a0) at ($ (ac) + (210:{\rada}) $) {};
  \coordinate (a1) at ($ (ac)!1!{(-360/3)}:(a0) $) {};
  \coordinate (a2) at ($ (ac)!1!{(-360/3)}:(a1) $) {};
  \coordinate (ae) at ($ (a1)!0.5!(a2) $) {};

  \coordinate (bc) at ($ (ac)!{\apoab}!(ae) $) {};  
  \coordinate (b0) at ($ (a1) $) {};
  \coordinate (b1) at ($ (bc)!1!{(-360/4)}:(b0) $) {};
  \coordinate (b2) at ($ (bc)!1!{(-360/4)}:(b1) $) {};
  \coordinate (b3) at ($ (bc)!1!{(-360/4)}:(b2) $) {};
  \coordinate (be) at ($ (b2)!0.5!(b3) $) {};

  \coordinate (cc) at ($ (bc)!{\apobc}!(be) $) {};
  \coordinate (c0) at ($ (b3) $) {};
  \coordinate (c1) at ($ (cc)!1!{(-360/5)}:(c0) $) {};
  \coordinate (c2) at ($ (cc)!1!{(-360/5)}:(c1) $) {};
  \coordinate (c3) at ($ (cc)!1!{(-360/5)}:(c2) $) {};
  \coordinate (c4) at ($ (cc)!1!{(-360/5)}:(c3) $) {};
  \coordinate (ce) at ($ (c1)!0.5!(c2) $) {};

  \coordinate (dc) at ($ (cc)!{\apocd}!(ce) $) {};
  \coordinate (d0) at ($ (c1) $) {};
  \coordinate (d1) at ($ (dc)!1!{(-360/6)}:(d0) $) {};
  \coordinate (d2) at ($ (dc)!1!{(-360/6)}:(d1) $) {};
  \coordinate (d3) at ($ (dc)!1!{(-360/6)}:(d2) $) {};
  \coordinate (d4) at ($ (dc)!1!{(-360/6)}:(d3) $) {};
  \coordinate (d5) at ($ (dc)!1!{(-360/6)}:(d4) $) {};
  \coordinate (de) at ($ (d4)!0.5!(d3) $) {};

  \coordinate (ec) at ($ (dc)!{\apode}!(de) $) {};
  \coordinate (e0) at ($ (d4) $) {};
  \coordinate (e1) at ($ (ec)!1!{-360/7}:(e0) $) {};
  \coordinate (e2) at ($ (ec)!1!{-360/7}:(e1) $) {};
  \coordinate (e3) at ($ (ec)!1!{-360/7}:(e2) $) {};
  \coordinate (e4) at ($ (ec)!1!{-360/7}:(e3) $) {};
  \coordinate (e5) at ($ (ec)!1!{-360/7}:(e4) $) {};
  \coordinate (e6) at ($ (ec)!1!{-360/7}:(e5) $) {};

  \path[draw, green, dashdotted] (b1) -- (b3);
  \path[draw, blue, dashed] (b0) -- (b2);

  \path[draw, red] (c0) -- (c3);
  \path[draw, green, dashdotted] (c4) -- (c1) -- (c3);
  \path[draw, blue, dashed] (c0) -- (c2) -- (c4);

  \path[draw, green, dashdotted] (d1) -- (d3) -- (d5) -- (d1);
  \path[draw, blue, dashed] (d0) -- (d2) -- (d4) -- (d0);

  \path[draw, red] (e0) -- (e5);
  \path[draw, green, dashdotted] (e6) -- (e1) -- (e3) -- (e5);
  \path[draw, blue, dashed] (e0) -- (e2) -- (e4) -- (e6);

  \path[draw, red, thick] (a0) -- (a1) -- (a2) -- cycle;
  \node[label=center:$0$] at ($(a0)!.25!(ac)$) {};
  \node[label=center:$1$] at ($(a1)!.25!(ac)$) {};
  \node[label=center:$2$] at ($(a2)!.25!(ac)$) {};

  \path[draw, red, thick] (b0) -- (b1) -- (b2) -- (b3) -- cycle;
  \node[label=center:$0$] at ($(b0)!.25!(bc)$) {};
  \node[label=center:$1$] at ($(b1)!.25!(bc)$) {};
  \node[label=center:$2$] at ($(b2)!.25!(bc)$) {};
  \node[label=center:$3$] at ($(b3)!.25!(bc)$) {};

  \path[draw, red, thick] (c0) -- (c1) -- (c2) -- (c3) -- (c4) -- cycle;
  \node[label=center:$0$] at ($(c0)!.25!(cc)$) {};
  \node[label=center:$1$] at ($(c1)!.25!(cc)$) {};
  \node[label=center:$2$] at ($(c2)!.25!(cc)$) {};
  \node[label=center:$3$] at ($(c3)!.25!(cc)$) {};
  \node[label=center:$4$] at ($(c4)!.25!(cc)$) {};

  \path[draw, red, thick] (d0) -- (d1) -- (d2) -- (d3) -- (d4) -- (d5) -- cycle;
  \node[label=center:$0$] at ($(d0)!.25!(dc)$) {};
  \node[label=center:$1$] at ($(d1)!.25!(dc)$) {};
  \node[label=center:$2$] at ($(d2)!.25!(dc)$) {};
  \node[label=center:$3$] at ($(d3)!.25!(dc)$) {};
  \node[label=center:$4$] at ($(d4)!.25!(dc)$) {};
  \node[label=center:$5$] at ($(d5)!.25!(dc)$) {};

  \path[draw, red, thick] (e0) -- (e1) -- (e2) -- (e3) -- (e4) -- (e5) -- (e6) -- cycle;
  \node[label=center:$0$] at ($(e0)!.25!(ec)$) {};
  \node[label=center:$1$] at ($(e1)!.25!(ec)$) {};
  \node[label=center:$2$] at ($(e2)!.25!(ec)$) {};
  \node[label=center:$3$] at ($(e3)!.25!(ec)$) {};
  \node[label=center:$4$] at ($(e4)!.25!(ec)$) {};
  \node[label=center:$5$] at ($(e5)!.25!(ec)$) {};
  \node[label=center:$6$] at ($(e6)!.25!(ec)$) {};

\end{tikzpicture}
\end{center}
\caption{
  Coloring edges of a facial-square in Lemma~\ref{lem:square}.
}
\label{fig:square}
\end{figure}

%% file: fig_walk.tex
\begin{figure}[h]
  \begin{center}
\begin{tikzpicture}[>=latex]
\begin{scope}[scale=0.45]
  \begin{tiny}
    \node[anchor=base] (a1) at (-0.5, 0) {$W:$};
    \draw (1,0.15) -- (30,0.15);

  \coordinate (a1) at (1, 0) {};
  \coordinate (b1) at (2, 0) {};
  \coordinate (c1) at (3, 0) {};
  \coordinate (d1) at (4, 0) {};
  \coordinate (e1) at (5, 0) {};
  \coordinate (c2) at (6, 0) {};
  \coordinate (f1) at (7, 0) {};
  \coordinate (g1) at (8, 0) {};
  \coordinate (c3) at (9, 0) {};
  \coordinate (h1) at (10, 0) {};
  \coordinate (b2) at (11, 0) {};
  \coordinate (i1) at (12, 0) {};
  \coordinate (j1) at (13, 0) {};
  \coordinate (k1) at (14, 0) {};
  \coordinate (l1) at (15, 0) {};
  \coordinate (j2) at (16, 0) {};
  \coordinate (m1) at (17, 0) {};
  \coordinate (n1) at (18, 0) {};
  \coordinate (j3) at (19, 0) {};
  \coordinate (b3) at (20, 0) {};
  \coordinate (o1) at (21, 0) {};
  \coordinate (a2) at (22, 0) {};
  \coordinate (p1) at (23, 0) {};
  \coordinate (q1) at (24, 0) {};
  \coordinate (r1) at (25, 0) {};
  \coordinate (s1) at (26, 0) {};
  \coordinate (t1) at (27, 0) {};
  \coordinate (q2) at (28, 0) {};
  \coordinate (u1) at (29, 0) {};
  \coordinate (a3) at (30, 0) {};

  \node[anchor=base,fill=white] (sa1) at (a1) {$a$};
  \node[anchor=base,fill=white] (sb1) at (b1) {$b$};
  \node[anchor=base,fill=white] (sc1) at (c1) {$c$};
  \node[anchor=base,fill=white] (sd1) at (d1) {$d$};
  \node[anchor=base,fill=white] (se1) at (e1) {$e$};
  \node[anchor=base,fill=white] (sc2) at (c2) {$c$};
  \node[anchor=base,fill=white] (sf1) at (f1) {$f$};
  \node[anchor=base,fill=white] (sg1) at (g1) {$g$};
  \node[anchor=base,fill=white] (sc3) at (c3) {$c$};
  \node[anchor=base,fill=white] (sh1) at (h1) {$h$};
  \node[anchor=base,fill=white] (sb2) at (b2) {$b$};
  \node[anchor=base,fill=white] (si1) at (i1) {$i$};
  \node[anchor=base,fill=white] (sj1) at (j1) {$j$};
  \node[anchor=base,fill=white] (sk1) at (k1) {$k$};
  \node[anchor=base,fill=white] (sl1) at (l1) {$l$};
  \node[anchor=base,fill=white] (sj2) at (j2) {$j$};
  \node[anchor=base,fill=white] (sm1) at (m1) {$m$};
  \node[anchor=base,fill=white] (sn1) at (n1) {$n$};
  \node[anchor=base,fill=white] (sj3) at (j3) {$j$};
  \node[anchor=base,fill=white] (sb3) at (b3) {$b$};
  \node[anchor=base,fill=white] (so1) at (o1) {$o$};
  \node[anchor=base,fill=white] (sa2) at (a2) {$a$};
  \node[anchor=base,fill=white] (sp1) at (p1) {$p$};
  \node[anchor=base,fill=white] (sq1) at (q1) {$q$};
  \node[anchor=base,fill=white] (sr1) at (r1) {$r$};
  \node[anchor=base,fill=white] (ss1) at (s1) {$s$};
  \node[anchor=base,fill=white] (st1) at (t1) {$t$};
  \node[anchor=base,fill=white] (sq2) at (q2) {$q$};
  \node[anchor=base,fill=white] (su1) at (u1) {$u$};
  \node[anchor=base,fill=white] (sa3) at (a3) {$a$};


    \draw (sa1) edge[out=45,in=135] (sa2); \draw (sa2) edge[out=45,in=135] (sa3);
    \draw (sb1) edge[out=45,in=135] (sb2); \draw (sb2) edge[out=45,in=135] (sb3);
    \draw (sc1) edge[out=45,in=135] (sc2); \draw (sc2) edge[out=45,in=135] (sc3);
    \draw (sj1) edge[out=45,in=135] (sj2); \draw (sj2) edge[out=45,in=135] (sj3);
    \draw (sq1) edge[out=45,in=135] (sq2);

    \draw[thick,red]   (sa1) -- ($(a1)-(0,1.00)$) -- ($(e1)-(0,1.00)$) -- (se1); \node[red]   (p1) at ($(a1)!0.5!(e1)-(0,1.50)$) {$P_1$};
    \draw[thick,green] (sd1) -- ($(d1)-(0,1.25)$) -- ($(g1)-(0,1.25)$) -- (sg1); \node[green] (p2) at ($(d1)!0.5!(g1)-(0,1.75)$) {$P_2$};
    \draw[thick,red]   (sf1) -- ($(f1)-(0,1.50)$) -- ($(l1)-(0,1.50)$) -- (sl1); \node[red]   (p3) at ($(f1)!0.5!(l1)-(0,2.00)$) {$P_3$};
    \draw[thick,green] (sk1) -- ($(k1)-(0,1.75)$) -- ($(n1)-(0,1.75)$) -- (sn1); \node[green] (p4) at ($(k1)!0.5!(n1)-(0,2.25)$) {$P_4$};
    \draw[thick,red]   (sm1) -- ($(m1)-(0,2.00)$) -- ($(t1)-(0,2.00)$) -- (st1); \node[red]   (p5) at ($(m1)!0.5!(t1)-(0,2.50)$) {$P_5$};
    \draw[thick,green] (sr1) -- ($(r1)-(0,2.25)$) -- ($(a3)-(0,2.25)$) -- (sa3); \node[green] (p6) at ($(r1)!0.5!(a3)-(0,2.75)$) {$P_6$};
  \end{tiny}
\end{scope}
\end{tikzpicture}
\end{center}
\caption{
  A facial walk $W$.
  Above the walk, the laminar structure of repeated occurences of vertices.
  Below the walk, maximal simple $W$-blocks colored red and green.
}
\label{fig:walk}
\end{figure}
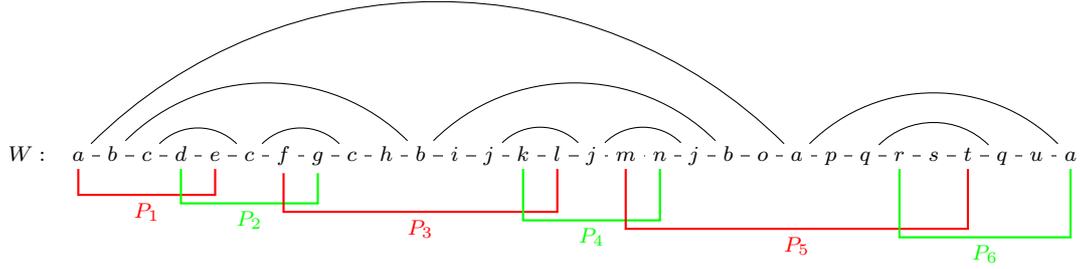

%% file: fig_face.tex
\begin{figure}[h]
  \begin{center}
\begin{tikzpicture}[>=latex]
  \begin{scope}[yscale=2,xscale=2]

  \tikzstyle{every node}=[red, circle,minimum size=5pt,inner sep=0pt,draw,fill,label distance=1pt]
  \tikzstyle{triangle}=[blue, regular polygon,regular polygon sides=4,minimum size=5pt,inner sep=0pt,draw,fill,label distance=1pt]

\coordinate (c) at (0,0) {};
\coordinate (s) at (0, -1) {};
\coordinate (t) at (0, 1) {};
\coordinate (l1) at ($ (c)!1!{1*(-360/10)}:(s) $) {};
\coordinate (l2) at ($ (c)!1!{2*(-360/10)}:(s) $) {};
\coordinate (l3) at ($ (c)!1!{3*(-360/10)}:(s) $) {};
\coordinate (l4) at ($ (c)!1!{4*(-360/10)}:(s) $) {};
\coordinate (r1) at ($ (c)!1!{1*( 360/10)}:(s) $) {};
\coordinate (r2) at ($ (c)!1!{2*( 360/10)}:(s) $) {};
\coordinate (r3) at ($ (c)!1!{3*( 360/10)}:(s) $) {};
\coordinate (r4) at ($ (c)!1!{4*( 360/10)}:(s) $) {};
\coordinate (s1) at (r1) {};
\coordinate (s12) at ($ (s)!1.5!{0.5*(180*8/10)/5}:(s1) $) {};
\coordinate (s2) at ($ (s)!1!{(180*8/10)/5}:(s1) $) {};
\coordinate (s23) at ($ (s)!1.5!{1.5*(180*8/10)/5}:(s1) $) {};
\coordinate (s3) at ($ (s)!1!{2*(180*8/10)/5}:(s1) $) {};
\coordinate (s34) at ($ (s)!1.5!{2.5*(180*8/10)/5}:(s1) $) {};
\coordinate (s4) at ($ (s)!1!{3*(180*8/10)/5}:(s1) $) {};
\coordinate (s45) at ($ (s)!1.5!{3.5*(180*8/10)/5}:(s1) $) {};
\coordinate (s5) at ($ (s)!1!{4*(180*8/10)/5}:(s1) $) {};
\coordinate (s56) at ($ (s)!1.5!{4.5*(180*8/10)/5}:(s1) $) {};
\coordinate (s6) at (l1) {};
\coordinate (t1) at (l4) {};
\coordinate (t2) at ($ (t)!1!{(180*8/10)/4}:(t1) $) {};
\coordinate (t23) at ($ (t)!1.5!{1.5*(180*8/10)/4}:(t1) $) {};
\coordinate (t3) at ($ (t)!1!{2*(180*8/10)/4}:(t1) $) {};
\coordinate (t34) at ($ (t)!1.5!{2.5*(180*8/10)/4}:(t1) $) {};
\coordinate (t4) at ($ (t)!1!{3*(180*8/10)/4}:(t1) $) {};
\coordinate (t45) at ($ (t)!1.5!{3.5*(180*8/10)/4}:(t1) $) {};
\coordinate (t5) at (r4) {};

    \draw (s) -- (l1); \draw (l4) -- (t) -- (r4); \draw (r1) -- (s);

    \draw[green] (r1) arc (-54:54:1);
    \draw[green] (l4) arc (126:234:1);

    \draw (t) -- (t2);
    \draw[green] (t2) .. controls (t23) .. (t3);
    \draw (t3) -- (t) -- (t4);
    \draw[green] (t4) .. controls (t34) and (t45) .. (t4);

    \draw (s) -- (s2);
    \draw[green] (s2) .. controls (s23) .. (s3);
    \draw (s) -- (s3);
    \draw (s) -- (s4);
    \draw[green] (s4) .. controls (s45) .. (s5);
    \draw (s) -- (s5);

    \draw[green] (r2) .. controls (t45) and (t34) .. (r2);
    \draw[green] (r2) .. controls (s12) and (s23) .. (r2);

    \node (sn) at (s) {};
    \node (tn) at (t) {};
    \node[triangle] (s1n) at (s1) {};
    \node[triangle] (s2n) at (s2) {};
    \node[triangle] (s3n) at (s3) {};
    \node[triangle] (s4n) at (s4) {};
    \node[triangle] (s5n) at (s5) {};
    \node[triangle] (s6n) at (s6) {};
    \node[triangle] (t1n) at (t1) {};
    \node[triangle] (t2n) at (t2) {};
    \node[triangle] (t3n) at (t3) {};
    \node[triangle] (t4n) at (t4) {};
    \node[triangle] (t5n) at (t5) {};
    \node[black,draw=none,fill=none] (fl) at (c) {$F$};
    \node[black,draw=none,fill=none] (sl) at ($(s)+(0, -0.25)$) {$s(F)$};
    \node[black,draw=none,fill=none] (tl) at ($(t)+(0, 0.25)$) {$t(F)$};
\end{scope}

\end{tikzpicture}
\end{center}
\caption{
  Filtering of list assignment of face $F$ in Lemma~\ref{lem:face}.
  Special vertices marked with red circles.
  Vertices in $A$ marked with blue squares.
  Walks $W_1,\ldots,W_k$ colored green.
}
\label{fig:face}
\end{figure}
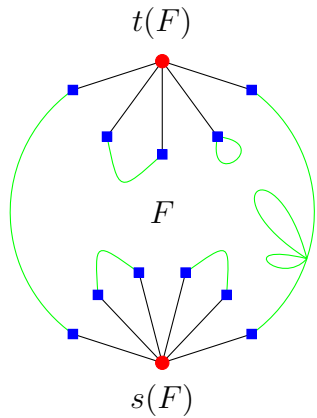